%% file: cablesfinal.tex
\documentclass[11pt]{amsart}
\usepackage{amsmath,amssymb,amsthm,graphicx, url, verbatim, float}
\usepackage{enumerate}
\usepackage{color, relsize}
\usepackage[margin=1.20in, marginpar = 2.5cm]{geometry}

\theoremstyle{definition}
\newtheorem{defn}{Definition}[section]
\newtheorem{remark}[defn]{Remark}

\theoremstyle{plain}
\newtheorem{thm}[defn]{Theorem}

\newtheorem{prop}[defn]{Proposition}
\newtheorem{lem}[defn]{Lemma} 
\newtheorem{cor}[defn]{Corollary}

\newtheorem*{namedtheorem}{\theoremname}
\newcommand{\theoremname}{testing}
\newenvironment{named}[1]{\renewcommand{\theoremname}{#1}\begin{namedtheorem}}{\end{namedtheorem}}

\usepackage{mathtools}

\newcommand{\writhe}{\mathrm{wr}}

\newcommand{\QQ}{{\mathbb{Q}}}

\begin{document}

\title{Crossing numbers of cable knots}

\author[E. Kalfagianni]{Efstratia Kalfagianni}
\author[R. Mcconkey]{Rob Mcconkey}
\address[]{Department of Mathematics, Michigan State University, East
Lansing, MI, 48824}
\email[]{kalfagia@msu.edu }

\address[]{Department of Mathematics and  Physics, 2200 Bonforte Blvd,
Colorado State University Pueblo,
\ \ \ \ \ \ \ \ \ \ \  \ \ \ \
Pueblo, CO 81001-4901}
\email[]{rob.mcconkey@csupueblo.edu}

\begin{abstract}We use the degree of the colored Jones knot polynomials to show that the crossing number of a $(p,q)$-cable of an adequate knot
with crossing number $c$ is  larger than $q^2\, c$. As an application we determine the crossing number of $2$-cables of adequate knots.

We also determine the crossing number of the connected sum of any adequate knot with a $2$-cable of an adequate knot.

\vskip 0.07in


 \end{abstract}
\thanks{The authors  acknowledge  partial research support through NSF Grants DMS-2004155 and DMS-2304033 and from the NSF/RTG grant DMS-2135960.}


\maketitle
\section{Introduction}
Given a knot $K$ we will use $c(K)$ to denote the crossing number of $K$, which is the smallest number of crossings over all diagrams that represent $K$.
Crossing numbers are known to be notoriously  intractable. For instance their behavior under
basic knot operations, such as connect sum of knots and  satellite operations, is poorly understood. In particular, the basic conjecture 
that if $K$ is a satellite knot with companion $C$ then $c(K)\geq c(C)$ is sill open  \cite[Problem 1.68]{Kirby}.
In this direction, Lackenby \cite{La2}  proved  that we have  $c(K)\geq 10^{-13}\, c(C)$, for any satellite knot $K$ with companion $C$.
In this note, we  prove a much stronger inequality for cables of adequate knots and we determine the exact crossing numbers of infinite families of such knots.
Since alternating knots are known to be adequate, our results apply, in particular, to cables of alternating knots.

 To state our results, 
for a knot $K$ in the 3-sphere let $N(K)$ denote a tubular neighborhood of $K$.
Given coprime integers $p,q$ let $K_{p, q}$ denote the $(p,q)$-cable of $K$.
In other words, $K_{p,q}$  is the simple closed curve on $\partial N(K)$ that wraps $p$ times around the meridian and $q$-times around the canonical longitude of $K$.
Recall
 that the
writhe of an adequate diagram $D=D(K)$ is an invariant of the knot $K$ \cite{Lickorishbook}.
We will use $\writhe(K)$ to denote this invariant.  

\begin{thm}\label{main} 
For any adequate knot $K$ with crossing number $c(K)$, and any coprime integers $p,q$,  we  have
 $c(K_{p,q})\geq q^2 \, c(K)+1$.
\end{thm}

Theorem \ref{main}, combined with the results of \cite{KLee}, has  applications in determining crossing numbers of prime satellite knots.
We have the following:

 \begin{cor} \label{exact}Let $K$ be an adequate knot with crossing number $c(K)$ and writhe  number $\writhe(K)$. 
 If $p=2\, \writhe(K)\pm 1$, then $K_{p,2}$ is non-adequate and $c(K_{p,2})=4\,  c(K)+1$.
 \end{cor}
 
 The proof of Corollary \ref{exact} shows that when $p=2\, \writhe(K)\pm 1$, if we apply the $(p, 2)$-cabling operation to an adequate diagram of $K$, the resulting diagram is a 
 minimum crossing diagram of the knot $c(K_{p,2})$. It should be compared with other
results in the literature asserting that the crossing numbers of some important classes of knots are realized by a ``special type" of knot diagrams. 
These classes include  alternating and more generally adequate knots, torus knots, Montesinos knots \cite{ Kauffmanstates, Murasugi, alternating} and
untwisted Whitehead doubles of adequate knots with zero writhe number \cite{KLee}. We note that these
Whitehead doubles and the cables $c(K_{p,2})$ of Corollary \ref{exact} are the first infinite families of prime satellite knots for which the crossing numbers have been determined.
In \cite{BMT2}, Baker Motegi and Takata obtained lower bounds
for crossing numbers of  Mazur doubles of adequate knots. In particular, they show that if $K$ is an adequate knot with $\mathrm{wr}(K)=0$, then the crossing number of the Mazur double of $K$ is either $9\, c(K)+2$ or $9\, c(K)+3$.

We note that a geometric lower
bound that applies to crossing number of satellites of hyperbolic knots is given in \cite{He}.

 Corollary \ref{exact}  allows us to compute the crossing number of $(\pm1, 2)$-cables  of adequate knots  that are equivalent to their mirror images (a.k.a. amphicheiral) since such knots are known have $\writhe(K)=0$. In particular, since for any adequate  knot $K$ with mirror image  $K^{*}$ the connect sum $K\# K^{*}$ is adequate and amphicheiral, we have the following:

\begin{cor} \label{mirror}For any adequate knot $K$ with crossing number $c(K)$ and mirror image $K^{*}$ let $K^{2}:=K\# K^{*}$.
Then $c(K^{2}_{\pm 1, 2})=8\, c(K)+1$.
\end{cor}

Our results also have an application to the open conjecture on the additivity of crossing numbers  \cite[Problem 1.68]{Kirby} under connect sums. 
Lower bounds for the connect sum of knots in terms of the crossing numbers of the summands that apply to all knots are obtained in \cite{La1, Ito}.
The conjecture has been proved in the cases where each summand is adequate \cite{ Kauffmanstates, Murasugi, alternating} or  a torus knots \cite{torus}, and when one summand is adequate and the other an untwisted Whitehead doubles of adequate knots with zero writhe number \cite{KLee}. To these we add the  following:

\begin{thm} \label{sum}
Suppose that $K$ is an adequate knot  and let
  $K_1:=K_{p, 2}$, where  $p=2\, \writhe(K)\pm 1$.
Then for any adequate knot $K_2$, the connected sum $K_1\# K_2$
is non-adequate and we have
  $$c(K_1\# K_2)=c(K_1)+c(K_2).$$
\end{thm}

It may be worth noting that out  of the 2977 prime knots with up to 12 crossings, 1851 are listed as adequate on Knotinfo  \cite{Knotinfo} and thus our results above can be applied to them.

\section{Crossing numbers of cables of adequate knots}

\subsection{Preliminaries}
A \emph{Kauffman state} on a knot diagram $D$ is a choice of either the $A$-resolution or the $B$-resolution for each  crossing of $D$ as shown in Figure \ref{f.kstate}.  The result of applying $\sigma$ to $D$ is a collection  $\sigma(D)$ of disjoint simple closed curves called \emph{state circles}.
The \emph{all-$A$} (resp.  \emph{all-$B$}) \emph{state}, denoted by $\sigma_A$ (resp.  $\sigma_B$) is the state where the $A$-resolution (resp.  the $B$-resolution)  is chosen at every crossing of $D$. 
\begin{figure}[H]
\def \svgwidth{.25\columnwidth}
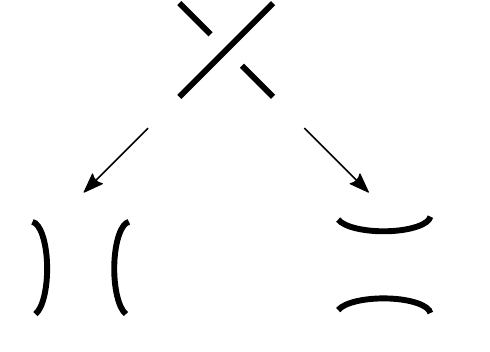
\caption{ \label{f.kstate} The $A$- and $B$-resolution and the corresponding edges of ${\mathbb G}_A(D)$ and ${\mathbb G}_B(D)$. } 
\end{figure}

\begin{itemize}
\item For an oriented  knot diagram $D$,  with $c(D)$  crossings, $c_+(D)$ and $c_-(D)$ are respectively the number of positive crossings and  negative crossings of  $D$ (see Figure \ref{f.pnc}). The \emph{writhe} of $D$, is given by $\writhe(D):=c_+(D) - c_-(D)$. 

\item The  graph ${\mathbb G}_A(D)$ (resp. ${\mathbb G}_B(D)$) has vertices the state circles of the  all-$A$ (resp.  all-$B$ state) and edges the segments recording the original location of the crossings (see Figure \ref{f.kstate}).
We denote by $v_A(D)$ (resp. $v_B(D)$)   the number of vertices of ${\mathbb G}_A(D)$ (resp. ${\mathbb G}_A(D)$).
\end{itemize}

\begin{figure}[H]
\def \svgwidth{.2\columnwidth}
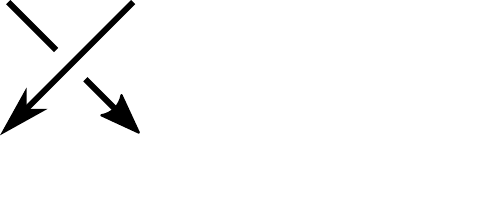
\caption{\label{f.pnc} A positive crossing and a negative crossing.}
\end{figure}  
\begin{defn}A knot diagram $D=D(K)$  is called  \emph{$A$-adequate}  (resp. \emph{$B$-adequate} )
 if ${\mathbb G}_A(D)$ (resp. ${\mathbb G}_B(D)$) 
  has no one-edged loops. 
  A knot is \emph{adequate} if it admits a diagram $D:=D(K)$   that is both $A$- and $B$-adequate {\cite{LickorishThistlethwaite, Lickorishbook}}. 
  \end{defn}

If $D:=D(K)$ is an adequate diagram the quantities $c(D)$, $c_{\pm}(D)$, $\mathrm{wr}(D)$
 are invariants of $K$ \cite{Lickorishbook}, and will be denoted by
$c(K)$, $c_{\pm}(K)$, $g_T(K)$, and $\mathrm{wr}(K)$ respectively.

Given a knot $K$ let   $J_K(n)$
denote its $n$-th  unreduced colored Jones polynomial, which is a Laurent polynomial in a variable $t$. The value on the unknot $U$ is given by
 $$J_U(n)(t) = (-1)^{n-1} \, \frac{t^{-n/2}-t^{n/2}}{t^{-1/2}-t^{1/2}},$$ for $n\geq 2$.
 Let  $d_+[J_K(n)]$ and $d_-[J_K(n)]$ denote the maximal and minimal degree of $J_K(n)$ in $t$, and set $$d[J_K(n)]:=4d_+[J_K(n)]-4d_-[J_K(n)].$$ 
For the purposes of this paper we will assume that
the set of cluster points  $$\left\{| n^{-2}\, d[J_K(n)]|  \right\}'_{n{\in \mathbb N}},$$ consists of a single point and denoted by $dj_K$. This number is called  the {\em Jones  diameter}  of $K$.
We recall the following.
\begin{thm}\cite{KLee}\label{clee}Let $K$ be a knot with Jones diameter $dj_K$ and crossing number $c(K)$. Then,
$$dj_K \leq 2\, c(K),$$ 
with equality $dj_K =2\, c(K)$
if and only if $K$ is adequate.

In particular, if  $K$ is a non-adequate knot admitting a diagram $D$ such that  $dj_K=2\, (c(D)-1)$,  then we have $c(D)=c(K)$.
\end{thm}
Next we recall a couple of results from the literature  that give the extreme degrees of the colored Jones polynomials of the cables $K_{p,q}$
 in the case where the degrees $ d_{\pm}[J_K(n)]$ are quadratic polynomials.

\begin{prop}\cite{BKT1, KTran} \label{BMTprop} Suppose that $K$ is a knot such that
 $ d_+[J_K(n)]=a_2\,  n^2+a_1\, n +a_0$  and  $ d_-[J_K(n)]=a_2^* \, n^2+a^*_1\, n +a^*_0$ are quadratic polynomials for all $n>0$.  Suppose, moreover,
  that $a_1\leq 0$,  $a^*_1\geq 0$ and that  $\frac{ p}{q}< 4\, a_2$,  $\frac{-p}{q}<- 4\, a_2^*$. 
  
 Then for $n$ large enough, we have

  $$4\, d_+[J_{K_{p,q}}(n)]=4\, q^2\, a_2\, n^2+( q\, 4\, a_1+2\, (q-1)\, (p-4\, q\, a_2))\,n+A,$$

  $$4\, d_-[J_{K_{p,q}}(n)]=4\, q^2\, a^*_2 \,n^2+(q\, 4\, a^*_1+2\, (q-1)\, (p-4\, q\, a^*_2)) \,n+A^*,$$
 where $A, A^*\in \QQ$ depend only on $K$ and $p,q$.
   \end{prop}
   \begin{proof}
 The first equation is shown in \cite{KTran} (see also \cite{BKT1}).
   To obtain the second equation, note that, since
   $K^*_{-p,q}=({K_{p,q}})^*$, we have
   $d_-[J_{K_{p,q}}(n)]=- d_+[J_{K^*_{-p,q}}(n)]$. Since $d_+[J_{K^{*}}(n)] =-d_-[J_K(n)]=-a_2^*\,  n^2-a^*_1\, n -a^*_0$, the result follows by applying
    the first equation to $K^*_{-p,q}$.
   
   \end{proof}
    Now we recall the second result promised earlier.
    
   \begin{lem}\cite{BKT1, KTran} \label{bigger}Let the notation and setting be as in Proposition \ref{BMTprop}.
   
     If $\frac{ p}{q}> 4a_2$, then
     $$4\, d_+[J_{K_{p,q}}(n)]=p\, q\, n^2+B,$$
where $B\in \QQ$ depends only on $K$ and $p,q$.

Similarly, if $\frac{-p}{q}>-4a_2^*$, then
$$4\, d_-[J_{K_{p,q}}(n)]=p\, q\, n^2+B',$$
where $B'\in \QQ$ depends only on $K$ and $p,q$.
   \end{lem}
   \begin{proof}
   
     The first equation is shown in \cite{KTran} (see also \cite{BKT1}).
   As in the proof of Proposition  \ref{BMTprop}, to see the second equation,
   we use the fact that
   $d_-[J_{K_{p,q}}(n)]=- d_+[J_{K^*_{-p,q}}(n)]$.
  Applying the first equation to $K^*_{-p,q}$,  we get $4\, d_+[J_{K^{*}_{-p,q}}(n)]=-p\, q\, n^2+B^{*},$
   and hence $4\, d_-[J_{K_{p,q}}(n)]=p\, q\, n^2-B^{*}$. Setting $B':=-B^{*}$ we obtain the desired result.
   \end{proof}

\subsection{ Lower bounds and admissible knots} We will say that a knot $K$ is {\em{admissible}} if there is a diagram $D=D(K)$ such that  we have
 $dj_K =2\, ( c(D)-1)$.
 Our interest in admissible knots comes from the fact that if $K$ is admissible and non-adequate, then by Theorem \ref{clee}, $D$ is a minimal diagram (i.e. $c(D)=c(K)$).

\begin{thm}\label{main1} 
Let $K$ be an adequate knot and let $c(K)$, $c_{\pm}(K)$ and $\writhe(K)$ be as above.
\begin{enumerate}[(a)]
\item For any coprime integers $p,q$, we have 
\begin{equation} \label{eq:ineq} c(K_{p,q})\geq q^2 \, c(K).\end{equation}
\item The cable
$K_{p,q}$ is admissible if and only if $q=2$ and $p=q\, \writhe(K)\pm 1$.
 \end{enumerate}
\end{thm}

\begin{proof}
Since
 $K$ is adequate we have
 
 $$4\, d_+[J_K(n)] =2\,  c_+(K)\,n^2 + O(n) \ \ {\rm and} \ \  4\, d_-[J_K(n)] =-2\,  c_-(K)\,n^2 + O(n),$$ and hence
 \begin{equation}
 4\, d_+[J_K(n)] - 4\, d_-[J_K(n)] = 2 \, c(K)\,n^2 + O(n).
 \label{eq:ad}
 \end{equation}
 for every $n\geq 0$ \cite{Lickorishbook}.
We distinguish three cases.
\vskip 0.05in 

{\bf Case 1.} Suppose that $\frac{ p}{q}< 2\, c_+(K)$ and $\frac{-p}{q}<2\, c_-(K)$.
 Then, $ d_+[J_K(n)]$ 
satisfies 
 the hypothesis of Proposition \ref{BMTprop} with $4\, a_2=2\, c_+(K)>0$ and
  $  d_-[J_K(n)]=- d_+[J_{K^{*}}(n)]$, where $d_+[J_{K^{*}}(n)]$ satisfies  that hypothesis  of Proposition \ref{BMTprop}  with $-4\, a^{*}_2=2\, c_+(K^{*})=2\, c_-(K)$. 
The requirements that $a_1\leq 0$ and  $a^{*}_1\geq 0$ are satisfied since for adequate knots the linear terms of the degree of $J_K^{*}(n)$ are multiples of Euler characteristics of spanning surfaces of $K$. Indeed, $a_1$ (resp.  $a^{*}_1$) is equal to (resp. the opposite of)  the Euler characteristic  of a surface bounded by $K$.
See \cite[Lemmas 3.6, 3.7]{KTran} or \cite{Indiana, EKVietnam}.
Now  Proposition \ref{BMTprop}  implies that, for sufficiently large $n$, the quadratic coefficient of   $d_{+}[J_{K_{p,q}}(n)]$ (resp  $d_{-}[J_{K_{p,q}}(n)]$) is equal to
 $4\, a_2=2\, c_+(K)$  (resp. $4\, a^{*}_2=-2\, c_-(K)$). Hence
the Jones diameter of ${K_{p,q}}$ is 

\begin{equation} \label{equal}
dj_{K_{p,q}} =2\, q^2\, c(K).
\end{equation}
 Now by Theorem \ref{clee} we get $c(K_{p,q})\geq q^2\, c(K)$ which proves part  (a) of
 Theorem \ref{main1} in this case. 
 
 For part (b), we recall that a diagram $D_{p,q}$ of $K_{p,q}$ is obtained as follows: Start with an adequate diagram $D=D(K)$ and take $q$ parallel copies to obtain a diagram $D^q$. In other words, take the $q$-cabling of $D$ following the blackboard framing. To obtain $D_{p,q}$ add  {\emph{$t$-twists}} to $D^q$,
 where $t:=p-q\, \writhe(K)$ as follows: If $t<0$ then a twist takes the leftmost string in $D^q$ and slides it over the $q-1$ strings to the right; then we repeat the operation
 $|t|$-times. If $t>0$ a twist takes the rightmost string in $D^q$ and slides it over the $q-1$ strings to the left; then we repeat the operation
 $|t|$-times.
 Now
 $$c({D_{p,q}})= q^2\, c(K)+|t|(q-1)= q^2\, c(K)+|p-q\, \writhe(K)|\, (q-1),$$ 
 while $dj_{K_{p,q}}=2\, q^2\, c(K)$.
  Now setting $2\, c({D_{p,q}})-2=dj_K$, we get $|p-q\, \writhe(K)|\, (q-1)=1$ which gives that  $q=2$ and $p=q\, \writhe(K)\pm 1$. Similarly, if we set $p=q\, \writhe(K)\pm 1$ and $q=2$, we find that $2\, c({D_{p,q}})-2=dj_{K_{p,q}}$ must also be true. 
  Hence in this case both (a) and (b) hold.
   \begin{figure}[h]
    \center
    \includegraphics[width=2.5cm]{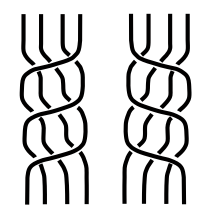}
    \caption{Three positive (left) and three negative (right) twists on four strands.  }
    \label{figureeight}
\end{figure}

{\bf Case 2.} Suppose that  $\frac{ p}{q}>2\, c_+(K)$.
 Then by  Lemma \ref{bigger},
 \begin{equation} \label{big} 
 4\, d_+[J_{K_{p,q}}(n)]=p\,q \,n^2+O(n).
 \end{equation}
Since  $\frac{ p}{q}>2\, c_+(K)$, multiplying both sides by $q^2$ we get
  \begin{equation} \label{big2} 
 p\, q>2q^2\, c_+(K).
 \end{equation}
 On the other hand, since $\frac{- p}{q}<0$, we clearly have $\frac{- p}{q}<2\, c_-(K)$, and 
 Proposition \ref{BMTprop}  applies to give
 \begin{equation} \label{big3} 
 4\, d_-[J_{K_{p,q}}(n)]=-2\, c_-(K)\,  n^2+O(n).
  \end{equation}
By Equations \eqref{big}, \eqref{big3} we obtain

\begin{equation} \label{big4}
4\, d_+[J_K(n)] - 4\, d_-[J_K(n)] =(p\, q+2\, q^2\, c_-(K))\, n^2 +O(n).
\end{equation}
Now by Equations \eqref{big4}, and \eqref{big2} we have,
\begin{equation}\label{big5}
dj_{K_{p,q}}=p\, q+2\,q^2\, c_-(K)>2\, q^2\, c_+(K)+2\, q^2\, c_-(K)=2\, q^2\, c(K),
\end{equation}
 which finishes the proof for part (a) of the theorem in this case.
  
Next we argue that  in this case, we don't get any admissible knots: First note that $$p>2q\,c_+(K)> q\, \writhe(K).$$
  As in Case 1 we get a diagram $D_{p,q}$ of $K_{p,q}$ 
  with
 $$c({D_{p,q}})= q^2\, c(K)+(p-q\, \writhe(K))\, (q-1),$$ 
 while $dj_{K_{p,q}}=p\, q+2\, q^2\, c_-(K)$.
  Now setting $2c({D_{p,q}})-2=dj_{K_{p,q}}$, and after some straightforward algebra, we find that
  in order for  $K_{p,q}$ to be admissible we must have
$$2\, (q^2-q)\, c_-(K)+2\, q\, c_+(K)+p\, (q-2)-2=0.$$
  However, since $p, c(K)>0$ and $q\geq 2$, above equation is never satisfied.

 \vskip 0.05in

{\bf Case 3.} Finally, suppose that $\frac{ -p}{q}>2c_-(K)>0$. By  Lemma \ref{bigger},
 \begin{equation} \label{big6} 
4\, d_-[J_{K_{p,q}}(n)]=p\, q\, n^2+O(n).
 \end{equation}
Since $\frac{ -p}{q}>2\, c_-(K)>0$, we conclude that 
 \begin{equation} \label{big7} 
 -p\,q>2\, q^2\,c_-(K).
 \end{equation}
Since $\frac{ p}{q}<0$, we clearly have $\frac{p}{q}<2\, c_+(K)$, and 
 Proposition \ref{BMTprop}  applies to give
 \begin{equation} \label{big8} 
 4\, d_+[J_{K_{p,q}}(n)]=2\, c_+(K)\, n^2+O(n).
  \end{equation}

By Equations \eqref{big6}, \eqref{big8}, and using \eqref{big7}, we obtain
\begin{equation}\label{big9}
dj_{K_{p,q}}=2\, q^2\, c_+(K)-p\, q>2\, q^2\, c_+(K)+2\, q^2\, c_-(K)=2\, q^2\, c(K),
\end{equation}
which finishes the proof for part (a) of the theorem.
An argument similar to this of Case 2 above shows that we don't get any admissible knots in Case 3 as well.
\end{proof}
\vskip 0.04in 
\begin{remark}In \cite{Stoi} inequality \eqref{eq:ineq} is also verified, for some choices of $p$ and $q$, using crossing number  bounds obtained from the ordinary Jones polynomial
in  \cite{Stong} and also from the 2-variable Kauffman polynomial. Theorem \ref{main} shows that the colored Jones polynomial and the results of \cite{KLee} provide better bounds for crossing numbers of satellite knots, allowing in particular exact computations.
\end{remark}
\section{Non-adequacy results}
 To prove the stronger version of inequality \eqref{eq:ineq}, stated in Theorem \ref{main}, we need to know that the cables $K_{p,q}$ are not adequate. This is the main result in this section.

  \begin{thm}\label{main2}Let $K$ be an adequate knot with crossing number $c(K)>0$ and suppose that $\frac{ p}{q}< 2\, c_+(K)$ and $\frac{-p}{q}<2\, c_-(K)$.
  Then, the cable $K_{p, q}$ is non-adequate.
  \end{thm}

To prove Theorem \ref{main2} we need the following lemma:

 \begin{lem}\label{ifadequate} Let $K$ be an adequate knot with crossing number $c(K)>0$ and suppose that
 $\frac{ p}{q}< 2\, c_+(K)$ and $\frac{-p}{q}<2\, c_-(K)$.
  If $K_{p, q}$ is adequate, then
 $c(K_{p, q})=q^2\, c(K)$.
  \end{lem}

 \begin{proof}  By Proposition \ref{BMTprop}, for $n$ large enough,
 $$
4\,  d_+[J_{(K_{p,q}}(n)]-4\, d_-[J_{K_{p,q}}(n)]=d_2 \,n^2+d_1 \, n+d_0,
$$
 with $d_i \in {\mathbb Q}$.
 By Proposition \ref{BMTprop}, and the discussion in the beginning of the proof of Theorem \ref{main1}, we compute $d_2=q^2\, (4\, a_2-4\, a^{*}_2)=2\, q^2\,c(K)$.
 Now if $K_{p,q}$ is adequate, since  by applying Equation \eqref{eq:ad} to  $K_{p, q}$ gives  $d_2=2 \, c(K_{p, q})$, we must have $c(K_{p, q})=q^2\,c(K)$.\end{proof}

We now give the proof of Theorem \ref{main2}:
\vskip 0.02in

 \begin{proof} First, we let $K$, $p$, and $q$ such that
 $t:=p-q\, \writhe(K)<0$.
 
Recall that if   $K$ has an adequate diagram $D=D(K)$ with $c(D)=c_+(D)+c_-(D)$ crossings and the all-$A$  (rep. all-$B$) resolution
 has $v_A=v_A(D)$ (resp.  $v_B=v_B(D)$) state circles, then 
 
  \begin{equation}
 4\, d_-[J_{K}(n)] =  -2\, c_- (D) n^2 + 2\, (c(D) -v_A(D))\,  n  +2\,  v_A(D) -2\,  c_+(D),
\label{top}
\end{equation}
 \begin{equation}
4 \, d_+[J_{K}(n)] = 2\, c_+ (D)\, n^2 + 2\, (v_B(D) - c(D))\,  n +2\,  c_-(D)-  2\, v_B(D).
\label{bottom}
\end{equation}
Equation (\ref{top}) holds for $A$-adequate diagrams $D=D(K)$.
Thus in particular the quantities $c_- (D), v_A(D)$ are invariants of $K$ (independent of the particular $A$-adequate diagram).
Similarly, Equation (\ref{bottom}) holds for $B$-adequate diagrams $D=D(K)$ and hence
$c_+ (D), v_B(D)$ are invariants of $K$. Recall also that $c(D) = c(K)$ since $D$ is adequate. 

Now we start with a knot $K$ that has an adequate diagram $D$.  Since $\mathrm{wr}(D)=\mathrm{wr}(K)$, we have $c_+ (D)=c_- (D)+\mathrm{wr}(K)$. Since $D$ is $B$-adequate
and $t<0$, the cable  $D_{p,q}$  is a $B$-adequate diagram of  $K_{p, q}$,  with $v_B(D_{p ,q})=q\,v_B(D)$ and $c_+(D_{p ,q})=q^2\, c_+(D)$. See Figure \ref{figureeight}.
Furthermore, since as said above these quantities are invariants of 
$K_{p,q}$, they remain the same for all $B$-adequate diagrams of
$K_{p ,q}$.
\begin{figure}[H]
    \center
    \includegraphics[width=8cm]{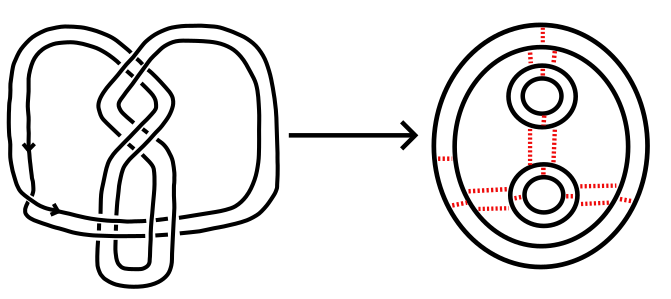}
    \caption{A diagram of the (-1,2)-cable of the figure eight knot and its all-$B$ state graph. }
    \label{figureeight}
\end{figure}

Now assume, for a contradiction, that $K_{p,q}$ is adequate: Then, it has a diagram ${\bar D}$ that is both $A$ and $B$-adequate.  
By above observation we must have
$v_B({\bar D})=v_B(D_{p,q})=q\,v_B(D)$ and $c_+({\bar D})=c_+(D_{p,q})=q^2 \,c_+(D)$.

 By Lemma \ref{ifadequate}, 
$c({\bar D})=c(K_{p,q})=q^2\,  c(K)$.
Write
$$
4\, d_+[J_{K_{p,q}}(n)]=x\, n^2+y\,n+z,
$$
for some  $x,y,z\in {\mathbb Q}$.

For sufficiently large $n$ we have two different expressions  for $x,y,z$. On one hand, because ${\bar D}$ is adequate, we can use  Equation (\ref{bottom}) to determine $x,y,z$.
On the other hand, using $4\, d_+[J_{K^*_{-p,q}}(n)]$,   $x,y,z$ can be determined using Proposition \ref{BMTprop} with $a_2$ and $a_1$ coming from Equation (\ref{bottom}).

We will use these two ways to find the quantity $y$.
Applying Equation (\ref{bottom})  to ${\bar D}$ we obtain
\begin{equation}
y=2\, (v_B({\bar D}-c({\bar D}) )) =2q \, v_B(D)-2\, q^2\, c(D)
\label{eq:33}
\end{equation}

On the other hand,  using  Proposition \ref{BMTprop} with $a_2$ and $a_1$ coming from Equation (\ref{bottom})  we have:
$4\,a_2=2\,c_+(D)=c(D)+\mathrm{wr}(K)$. Also, we have
$4\,a_1=2 \,v_B(D)-2\,c(D)$.
We obtain
\begin{equation}
y=q\, (4\, a_1)-2\, q\, (q-1)\, (4\, a_2)+2\, (q-1)\, p= 2\, q\, v_B(D)-2\, q^2\, c(D)+2\,(q-1)\,p-2\, q\, (q-1)\mathrm{wr}(K).
\label{eq:34}
\end{equation}

For the two expressions derived for $y$  from Equations \eqref{eq:33} and  \eqref{eq:34} to agree we must have 
$2\, q\,((q-1)\, 2\, \mathrm{wr}(K) + p) - 2\, p = 0$.
However this is impossible since $q>1$ and $p,q$ are coprime.
This contradiction shows that $K_{p,q}$ is  non-adequate.

To deduce the result for ${K_{p,q}}$, with $t(K, p, q):=p-q\, \writhe(K)>0$, let $K^{*}$ denote the mirror image of $K$. Note that  $(K_{p,q})^*=K^*_{-p,q}$ and since being adequate is a property that is preserved under taking mirror images, it is enough to show that $K^*_{-p,q}$ is non-adequate. Since $t(K^{*}, -p, q):=-p-q\, \writhe(K^{*})=-t(K, p, q)<0$, the later result follows from the argument above.
  \end{proof}

 \vskip 0.06in
  
Now we are ready to give the proofs of  Theorem \ref{main} and Corollary \ref{exact} which which we restate for the convenience of the reader:

 \begin{named}{Theorem \ref{main}}
 For any adequate knot $K$ with crossing number $c(K)$, and any coprime integers $p,q$,  we  have
 $c(K_{p,q})\geq q^2 \,c(K)+1$.
 
\end{named}
 \begin{proof}
 By Theorem \ref{main1}, we have
 $c(K_{p, q})\geq q^2\, c(K)$.
 We need to show that this inequality is actually strict. Following the proof of
  of Theorem \ref{main1} we distinguish three cases:
  \vskip 0.04in
 
 {\bf Case 1.}  Suppose that $\frac{ p}{q}< 2\, c_+(K)$ and $\frac{-p}{q}<2\, c_-(K)$. Then,
 by Equation \eqref{equal}, we have $dj_{K_{p,q}} =2\, q^2\, c(K)$.
 By  Theorem \ref{main2} $K_{p,q}$ is non-adequate and hence  by Theorem \ref{clee} again we have $2\, c(K_{p,q})> dj_{K_{p,q}}$, and  the strict inequality follows.
 
 \vskip 0.04in
 
 {\bf Case 2.}  Suppose  that If $\frac{ p}{q}>2\, c_+(K)$.  Then
 by Equation \eqref{big5}, we have
 $c(K_{p, q})> q^2\, c(K)$, and the result follows in this case
 \vskip 0.04in
 
 {\bf Case 3.} Suppose that $\frac{-p}{q}>2\, c_-(K)$.
 Then
 by Equation 
 \eqref{big9} again we have $c(K_{p, q})> q^2 \, c(K)$, as desired.
\end{proof}
 \vskip 0.1in
 
 Next we discuss  how to deduce Corollary \ref{exact}:
 \vskip 0.07in
 
  \begin{named}{Corollary  \ref{exact}}Let $K$ be an adequate knot with crossing number $c(K)$ and writhe  number $\writhe(K)$. 
 If $p=2\, \writhe(K)\pm 1$, then $K_{p,2}$ is non-adequate and $c(K_{p,2})=4\,  c(K)+1$.
 \end{named}
 \begin{proof}
If $q=2$ and $p=q\, \writhe(K)\pm 1$, then,  by Theorem~\ref{main1},
$K_{p,q}$ is admissible. Thus by Theorem \ref{clee}, the diagram
$D_{p,2}$ constructed in the proof of Theorem \ref{main1}
is minimal. That is
$c({K_{p,2}})=c({D_{p,2}})=4\, c(K)+1$. \end{proof}
\vskip 0.08in

\section{Composite non-adequate knots }  In this section we prove Theorem \ref{sum}.
  
Given a knot $K$, such that for $n$ large enough the degrees of the colored Jones polynomials
  of $K$ are quadratic polynomials with rational coefficients, we will write
   $$ 4\, d_+[J_{K}(n)]-4\, d_-[J_{K}(n)]=d_2(K )\, n^2+d_1(K)\,  n+d_0(K).$$
    
  \begin{lem} \label{notadequate2}Let $K$ be a non-trivial adequate knot, $p =2\, \writhe(K)\pm 1$
  and let  $K_1:=K_{p,2}$. Then for any adequate knot $K_2,$ the connected sum $K_1\# K_2$
 is non-adequate.
 \end{lem}

    \begin{proof} 
  The claim is proven by applying the arguments applied to $K_1=K_{p,2}$
   in the proofs
  of Lemma \ref{ifadequate} and Theorem  \ref{main2} to  $K_1\# K_2$ and properties of the degrees of colored Jones polynomial \cite[Lemma 5.9]{KLee}.
 
  First we claim that if $K_1\# K_2$ were adequate then we would have 
 \begin{align}
 \label{e.step}
c (K_1\# K_2)=4\, c(K)+c(K_2).
\end{align}
    
   Note that as $p=2\, \writhe(K)\pm 1$, we have $\frac{ p}{2}< 2\,c_+(K)$ and $\frac{-p}{2}<2\, c_-(K)$. Hence 
    Proposition \ref{BMTprop} applies to $K_1$.
    Now write
  $$ 4\,d_+[J_{K_1\#K_2}(n)]-4\,d_-[J_{K_1\#K_2}(n)]=d_2(K_1\#K_2) \, n^2+d_1(K_1\#K_2) \,n+d_0(K_1\#K_2).$$

 Since we assumed that  $K_1\# K_2$  is adequate, we have
 $d_2(K_1\# K_2) =2\, c(K_1 \# K_2)$ and by
 \cite[Lemma 5.9]{KLee}
  $d_2(K_1\# K_2)=d_2(K_1)+d_2(K_2)=2 \, 4\, c(K)+2\,  c(K_2)$ which leads to \eqref{e.step}.
 \vskip 0.04in

{\bf Case 1.}  Suppose that $p-2\, \writhe(K)=-1 <0$.
  
  Start with $D=D(K)$ an adequate diagram and let $D_1:=D_{p, 2}$ be constructed as in the proof of Theorem \ref{main1}.  Also let $D_2$ be an adequate diagram of $K_2$.
As in the proof of  Theorem \ref{main2} conclude that $D_1\# D_2$ is a $B$-adequate diagram for $K_1\# K_2$ and that the quantities
$v_B(D_1\#D_2)=2\, v_B(D)+v_B(D_2)-1$ and $c_+(D_1\#D_2) =4\, c_+(D)+c_+(D_2)$ are invariants of $K_1\# K_2$.

Let ${\bar D}$ be an adequate diagram. Then
$$
v_B({\bar D})=v_B(D_1\#D_2)=2\,v_B(D)+v_B(D_2) -1\ { \rm{and}}\ 
c_+({\bar D})=4\, c_+(D)+c_+(D_2).
$$

  Next we will calculate the quantity  $d_1(K_1\# K_2)$ in  two ways:
    Firstly, since we assumed that ${\bar D}$ is an adequate diagram for $K_1\# K_2$,  applying  Equation \eqref{bottom},  we get 
    $$d_1(K_1\# K_2)=2\, (v_B({\bar D})-c({\bar D}) )=2\, (2\, v_B(D)+v_B(D_2) -1-4 \, c(D)-c(D_2)).$$
       
Secondly, using by Proposition  \ref{BMTprop} we get
   $d_1(K_1)=2\,  (2\, v_B(D)- 4 \, c(D)+p+2\, \mathrm{wr}(K))$.
  Thus we get
   $$d_1(K_1\# K_2)=d_1(K_1)+d_1( K_2)-2=2\,  (2\, v_B(D)- 4\, c(D)+p-2\, \mathrm{wr}(K)+v_B(D_2) - c(D_2)-1).$$
   
Now note that in order for the two resulting expressions for  $d_1(K_1\# K_2)$ to be equal we must have
   $(p-2\, \mathrm{wr}(K))=0$ which contradicts our assumption that $p-2\, \mathrm{wr}(K)=-1$.
    We conclude that $K_1\#K_2$ is non-adequate.

  \vskip 0.02in

{\bf Case 2.} 
  Assume now that  $p-2\, \mathrm{wr}(K)=1$.  
 Since $(K_{p,2})^*=K^*_{-p,2}$ and  being adequate  is preserved under taking mirror images, it is enough to show that $K^*_{-p,2}\#K^*_2$ is non-adequate. 
 Since $-p-2\, \writhe(K^{*})=-(p - 2\, \mathrm{wr}(K)))=-1$, the later result follows from Case 1.

     \end{proof}

  Now we give the proof of Theorem \ref{sum}, which we also restate here:

\begin{named}{Theorem  \ref{sum}}
Suppose that $K$ is an adequate knot  and let
  $K_1:=K_{p, 2}$, where  $p=2\, \writhe(K)\pm 1$.
Then for any adequate knot $K_2$, the connected sum $K_1\# K_2$
is non-adequate and we have
  $$c(K_1\# K_2)=c(K_1)+c(K_2).$$
\end{named}
\begin{proof}Note that if $K$ is the unknot then so is $K_{p,2}$ and the result follows trivially.
  Suppose that $K$ is a non-trivial knot. Then, by Lemma \ref{notadequate2}, we obtain that $K_1\#K_2$ is non-adequate.
  By Part (b) of Theorem \ref{main1}, we have $dj_{K_1}=2\, (c(D_{\pm 1, 2})-1)$ and $dj_{K_2}=2\, c(D_2)=2\, c(K)$ where $D_2$ is an adequate diagram for $K_2$.
  Hence, 
$dj_{K_1\#K_2}=2\, (c(D_1\#D_2)-1)$, where $ D_1=D_{\pm1, 2}$ and
 by Theorem \ref{clee},
$$c(K_1\#K_2)=c(D_1\#D_2)=c(D_1)+c(D_2)=c(K_1)+c(K_2),$$ where the last equality follows since, by Corollary \ref{exact}, we have  $c(K_1)=c(D_1)=c(D_{p, 2}).$
  \end{proof}


\bibliographystyle{plain}
\bibliography{references}

\end{document}

%% file: resolution.pdf_tex
\begingroup%
  \makeatletter%
  \providecommand\color[2][]{%
    \errmessage{(Inkscape) Color is used for the text in Inkscape, but the package 'color.sty' is not loaded}%
    \renewcommand\color[2][]{}%
  }%
  \providecommand\transparent[1]{%
    \errmessage{(Inkscape) Transparency is used (non-zero) for the text in Inkscape, but the package 'transparent.sty' is not loaded}%
    \renewcommand\transparent[1]{}%
  }%
  \providecommand\rotatebox[2]{#2}%
  \newcommand*\fsize{\dimexpr\f@size pt\relax}%
  \newcommand*\lineheight[1]{\fontsize{\fsize}{#1\fsize}\selectfont}%
  \ifx\svgwidth\undefined%
    \setlength{\unitlength}{235.17197737bp}%
    \ifx\svgscale\undefined%
      \relax%
    \else%
      \setlength{\unitlength}{\unitlength * \real{\svgscale}}%
    \fi%
  \else%
    \setlength{\unitlength}{\svgwidth}%
  \fi%
  \global\let\svgwidth\undefined%
  \global\let\svgscale\undefined%
  \makeatother%
  \begin{picture}(1,0.71700104)%
    \lineheight{1}%
    \setlength\tabcolsep{0pt}%
    \put(0,0){\includegraphics[width=\unitlength,page=1]{resolution.pdf}}%
    \put(-0.01024017,-0.01014027){\color[rgb]{0,0,0}\makebox(0,0)[lt]{\lineheight{1.25}\smash{\begin{tabular}[t]{l}$A$-resolution\end{tabular}}}}%
    \put(0.62121308,-0.01014027){\color[rgb]{0,0,0}\makebox(0,0)[lt]{\lineheight{1.25}\smash{\begin{tabular}[t]{l}$B$-resolution\end{tabular}}}}%
    \put(0,0){\includegraphics[width=\unitlength,page=2]{resolution.pdf}}%
  \end{picture}%
\endgroup%

%% file: posnegcrossing.pdf_tex
\begingroup%
  \makeatletter%
  \providecommand\color[2][]{%
    \errmessage{(Inkscape) Color is used for the text in Inkscape, but the package 'color.sty' is not loaded}%
    \renewcommand\color[2][]{}%
  }%
  \providecommand\transparent[1]{%
    \errmessage{(Inkscape) Transparency is used (non-zero) for the text in Inkscape, but the package 'transparent.sty' is not loaded}%
    \renewcommand\transparent[1]{}%
  }%
  \providecommand\rotatebox[2]{#2}%
  \newcommand*\fsize{\dimexpr\f@size pt\relax}%
  \newcommand*\lineheight[1]{\fontsize{\fsize}{#1\fsize}\selectfont}%
  \ifx\svgwidth\undefined%
    \setlength{\unitlength}{229.15473517bp}%
    \ifx\svgscale\undefined%
      \relax%
    \else%
      \setlength{\unitlength}{\unitlength * \real{\svgscale}}%
    \fi%
  \else%
    \setlength{\unitlength}{\svgwidth}%
  \fi%
  \global\let\svgwidth\undefined%
  \global\let\svgscale\undefined%
  \makeatother%
  \begin{picture}(1,0.42056259)%
    \lineheight{1}%
    \setlength\tabcolsep{0pt}%
    \put(0,0){\includegraphics[width=\unitlength,page=1]{posnegcrossing.pdf}}%
    \put(0.04168163,0.01538434){\makebox(0,0)[lt]{\lineheight{1.25}\smash{\begin{tabular}[t]{l}$+1$\end{tabular}}}}%
    \put(0,0){\includegraphics[width=\unitlength,page=2]{posnegcrossing.pdf}}%
    \put(0.74862816,0.01538434){\makebox(0,0)[lt]{\lineheight{1.25}\smash{\begin{tabular}[t]{l}$-1$\end{tabular}}}}%
  \end{picture}%
\endgroup%